\documentclass[12pt]{article}
\usepackage[utf8]{inputenc}
\usepackage[margin = 1 in]{geometry}

\usepackage{amsmath}
\usepackage{amsthm}
\usepackage{amssymb}
\usepackage{mathrsfs}
\usepackage{mathtools,alltt,changepage}
\usepackage{tikz}
\usepackage{tikz-cd}
\usepackage{mathtools}
\usepackage{verbatim}
\usepackage{appendix}
\usepackage{xcolor}
\usepackage{longtable}
\usepackage{url,todonotes}
\usepackage{hyperref}
\usepackage{cleveref}
\hypersetup{colorlinks=true,citecolor=blue,urlcolor =black,linkbordercolor={1 0 0}}

\newtheorem{theorem}{Theorem}

\newtheorem{lemma}[theorem]{Lemma}

\newtheorem{question}[theorem]{Question}

\theoremstyle{definition}
\newtheorem{defn}[theorem]{Definition}
\crefname{defn}{Definition}{Definitions}

\theoremstyle{remark}
\newtheorem*{remark}{Remark}

\newcommand{\on}{\operatorname}
\newcommand{\mb}{\mathbb}

\renewcommand{\cong}{\simeq}

\title{A note on $r$ hypersurfaces intersecting in $\mb{P}^r$}
\author{Dennis Tseng}
\date{\today}

\begin{document}

\maketitle

\vspace*{-0.3in}
\begin{abstract}
We consider the locus of $r$-tuples of homogeneous forms of some fixed degree whose common vanishing locus in $\mb{P}^r$ is positive dimensional. We show that any component of maximal dimension of that locus either consists of homogeneous forms all vanishing on some line or homogeneous forms where a proper subset fail to intersect properly. 
\end{abstract}

\section{Introduction} 
A general choice of $r$ hypersurfaces in $\mb{P}^r$ will intersect in finitely many points. Given a choice of degrees $d_1\leq \cdots\leq d_r$, we can consider the closed locus 
\begin{align*}
Z:=\{(F_1,\ldots,F_r)\mid \{F_1=\cdots=F_r=0\}\text{ is positive dimensional}\}    
\end{align*}
inside the space $\mb{A}^{\binom{r+d_1}{d_1}}\times\cdots\times\mb{A}^{\binom{r+d_r}{d_r}}$ of all $r$ tuples of homogeneous forms. 

Previously, the author showed that if $d_i\leq d_1+\binom{d_1}{2}(i-1)$ and $d_1\geq 2$, then the unique component of maximal dimension consists of homogeneous forms all vanishing on some line \cite[Theorem 1.3]{TCIH}. The method was then applied to get partial results on the Kontsevich space of rational curves on hypersurfaces of degree $n-1$ in $\mb{P}^n$ \cite{TRocky} and on the space of hypersurfaces with positive dimensional singular locus \cite[Theorem 1.6]{TCIH} extending work of Slavov \cite{K15}. 

The purpose of this note is to give the following result about $Z$ that holds for all choices of degrees. 

\begin{theorem}
\label{MT}
For all choices of degrees $d_1\leq \cdots \leq d_r$, the unique component of $Z$ of maximal dimension not contained in
\begin{align}
\label{dim2}
\{(F_1,\ldots,F_r)\mid \dim(\{F_1=\cdots=F_{r-1}=0\})\geq 2\} 
\end{align}
consists of homogeneous forms all vanishing on some line.
\end{theorem}

While the statement of \Cref{MT} is qualitative, the proof uses the same methods as \cite{TCIH} and involves some numerical bounds working out right. Finally, like in \cite{TCIH}, the proof also applies for the locus of $r+a-1$ homogeneous forms whose common vanishing locus is positive dimensional for $a\geq 1$, and we state this slightly more general form in \Cref{MT2}.

\section{Further questions}
Given \Cref{MT}, one might ask for the largest components of the locus \eqref{dim2}. More precisely, we can ask
\begin{question}
\label{QQ}
Fix $k\leq r$ and $d_1\leq \cdots \leq d_k$. Does a component of maximal dimension of
\begin{align}
\label{locus}
\{(F_1,\ldots,F_k)\mid \dim(\{F_1=\cdots=F_k=0\})\geq r-k+1\}\subset \prod_{i=1}^{k}\mb{A}^{\binom{r+d_i}{d_i}}
\end{align}
not contained in 
\begin{align*}
    \{(F_1,\ldots,F_k)\mid \dim(\{F_1=\cdots=F_{k-1}=0\})\geq r-k+2\}\times \mb{A}^{\binom{r+d_k}{d_k}}\subset \prod_{i=1}^{k}\mb{A}^{\binom{r+d_i}{d_i}}
\end{align*}
consist of homogeneous forms all vanishing on some dimension $r-k+1$ linear space?
\end{question}

\Cref{MT} says the answer to \Cref{QQ} is always yes when $k=r$. When $k<r$ this is no longer true. For example, fix $k=2$ and $(d_1,d_2)=(2,2)$. One can quickly verify that \eqref{locus} has two components for every $r\geq 2$: the locus where $F_1$ and $F_2$ both vanish on some hyperplane and the locus where $F_1$ and $F_2$ both vanish on some quadric hypersurface. 

Also, by setting up an incidence correspondence, we check the two components have codimensions $r^2$ and $\frac{r^2+3r}{2}$ respectively in $\mb{A}^{\binom{r+2}{2}}\times \mb{A}^{\binom{r+2}{2}}$.
\begin{center}
    \begin{tabular}{c|c|c|c}
     & $r=2$ & $r=3$ & $r=4$\\\hline
    $\{F_1=F_2=0\}$ contains hyperplane & 4 & 9 & 16\\\hline
    $\{F_1=F_2=0\}$ contains quadric & 5 & 9 & 14 
    \end{tabular}
\end{center}
In particular, we see in this case for $r=2$ the the answer to \Cref{QQ} is yes as predicted by \Cref{MT}, for $r=3$ the answer is still yes but the largest component is no longer unique, and for $r\geq 4$ the answer is no. Given this example, one can also ask the following preliminary question where we set all the $d_i$'s to be equal. 
\begin{question}
\label{QQ2}
Fix $k\leq r$ and a degree $d$. Is it true that a component of maximal dimension of
\begin{align*}
    \{(F_1,\ldots,F_k)\mid \dim(\{F_1=\cdots=F_k=0\})\geq r-k+1\}\subset (\mb{A}^{\binom{r+d}{d}})^k
\end{align*}
must consist either of homogeneous forms vanishing on some dimension $r-k+1$ linear space or homogeneous forms that are linearly dependent?
\end{question}
If we fix $k\leq r$ and let $d>>0$, the author can show that there is a unique largest component and the first possibility occurs as a special case of a more general problem \cite{TVB}. If $k=r$, then again the locus of homogeneous forms vanishing on some line is the unique component of maximal dimension either by \cite[Theorem 1.3]{TCIH} or by applying \Cref{MT} and permuting the hypersurfaces.

\subsection{Acknowledgements}
The author would like to thank his advisor Joe Harris for helpful conversations. 

\section{Definitions}
We will follow the notation in \cite{TCIH}, since we will rely fundamentally on its main argument. As a trade off, the notation will be more cumbersome. Finally, the reader is referred to \cite[Section 2]{TCIH} for a worked example of the key argument, without the notational baggage. We will work over an algebraically closed field of arbitrary characteristic. 
\begin{defn}[{\cite[Definition 3.7]{TCIH}}]
Let $W_{r,d}\cong \mb{A}^{\binom{r+d}{d}}$ be the affine space whose underlying vector space is $H^0(\mb{P}^r,\mathscr{O}_{\mb{P}^r}(d))$. 
\end{defn}

\begin{defn}[{\cite[Definition 3.9]{TCIH}}]
\label{incidencenotation}
Given a tuple $(d_1,\ldots,d_k)$ of positive integers, $a$ a positive integer, and a subscheme $X\subset \mb{P}^r$, define
\begin{align*}
    \Phi^{\mb{P}^r,a}_{d_1,\ldots,d_k}(X)\subset \prod_{i=1}^{k}W_{r,d_i}
\end{align*}
to be the locus of tuples $(F_1,\ldots,F_k)$ of homogeneous forms of degrees $(d_1,\ldots,d_k)$ such that the vanishing locus $\{F_1=\cdots=F_k=0\}\cap X$ has dimension at least $\dim(X)-k+a$. 
\end{defn}

\subsection{Definitions used in proof}
For the proof of \Cref{MT2}, we will also need the following definitions
\begin{defn}[{\cite[Definitions 3.14 and 4.1]{TCIH}}]
Given $b\leq r$, let
\begin{align*}
\Phi^{\mb{P}^r,a}_{d_1,\ldots,d_k}(X,\on{Span}(r,b))\subset \Phi^{\mb{P}^r,a}_{d_1,\ldots,d_k}(\mb{P}^r)
\end{align*}
be the locus of forms $(F_1,\ldots,F_k)$ such that $\{F_1=\cdots=F_k=0\}\cap X$ contains an integral subscheme of dimension $\dim(X)-k+a$ with span exactly a dimension $b$ plane. 
\end{defn}
The locus $\Phi^{\mb{P}^r,a}_{d_1,\ldots,d_k}(X,\on{Span}(r,b))$ is a constructible subset given the discussion around \cite[Definitions 3.14 and 4.1]{TCIH}. 

\begin{defn}[{\cite[Definition 4.1]{TCIH}}]
Define $h_{r,a}(d):=(r-a)\binom{d+a-1}{d-1}+\binom{d+a}{d}$.
\end{defn}

The following elementary lemma can be found in \cite[Theorem 1.3]{P15}, for example, and it is proven by taking hyperplane slices and applying \cite[Lemma 3.1]{H82}. 
\begin{lemma}
\label{nondeg}
If $Z\subset \mb{P}^r$ is a nondegenerate integral scheme of dimension $a$, then the Hilbert function $h_Z(d)$ of $Z$ is bounded below by $h_{r,a}(d)$. 
\end{lemma}

\section{Proof of Main Theorem}
\Cref{MT} follows from \Cref{MT2} when $a=1$. We will prove \Cref{MT2}. 
\begin{theorem}
\label{MT2}
If $a\geq 1$ and $1\leq d_1\leq \cdots \leq d_{r+a-1}$ are integers, the unique component of maximal dimension of 
\begin{align*}
    \Phi^{\mb{P}^r,a}_{d_1,\ldots,d_{r+a-1}}(\mb{P}^r)\backslash \left(\Phi^{\mb{P}^r,a}_{d_1,\ldots,d_{r+a-2}}(\mb{P}^r)\times W_{r,d_{r+a-1}}\right)
\end{align*}
consists of $r+a-1$ tuples $(F_1,\ldots,F_{r+a-1})$ of homogeneous forms of degree $(d_1,\ldots,d_{r+a-1})$ where $\{F_1=\cdots=F_{r+a-1}=0\}$ contains some line. 
\end{theorem}

The key input from \cite{TCIH} is \Cref{technical}.
\begin{lemma}
\label{technical}
We have the codimension of 
\begin{align}
\label{tlb}
    \Phi^{\mb{P}^r,a}_{d_1,\ldots,d_k}(\mb{P}^r,\on{Span}(r,b))\backslash \left(\Phi^{\mb{P}^r,a}_{d_1,\ldots,d_{k-1}}(\mb{P}^r)\times W_{r,d_k}\right)
\end{align} 
in $\prod_{i=1}^{k}{W_{r,d_i}}$ is at least
\begin{align*}
    -\dim(\mb{G}(b,r))+ \min\left\{\sum_{j=1}^{r-b+a}h_{b,b-i_j+j}(d_{i_j})\mid 1\leq i_1<\cdots<i_{r-b+a}= k\right\}.
\end{align*}
\end{lemma}

\begin{proof}[Proof of \Cref{technical}]
This follows from the proof of \cite[Lemma 4.2]{TCIH}. Since there are no new ideas not contained in \cite{TCIH}, we will give an informal proof, recapping the key idea in \cite{TCIH} and mentioning the small modification necessary to show \Cref{technical}. 

First, we can reduce to the case $b=r$ by a standard incidence correspondence. The locus
\begin{align}
\label{incidenceequation}
    \left\{\Lambda,(F_1,\ldots,F_k)\mid (F_1|_{\Lambda},\ldots,F_k|_{\Lambda})\in     \Phi^{\mb{P}^r,a+r-b}_{d_1,\ldots,d_k}(\Lambda,\on{Span}(r,b))\backslash \left(\Phi^{\mb{P}^r,a+r-b}_{d_1,\ldots,d_{k-1}}(\Lambda)\times W_{r,d_k}\right)\right\}
\end{align}
inside of $\mb{G}(b,r)\times \prod_{i=1}^{k}{W_{r,d_i}}$ surjects onto \eqref{tlb} when forgetting the $\mb{G}(b,r)$ factor. When we forget the $\prod_{i=1}^{k}{W_{r,d_i}}$ factor, the fibers of the projection of \eqref{incidenceequation} to $\mb{G}(b,r)$ are all isomorphic and all have codimension in $\prod_{i=1}^{k}{W_{r,d_i}}$ equal to the codimension of 
\begin{align}
\label{ie2}
        \Phi^{\mb{P}^b,a+r-b}_{d_1,\ldots,d_k}(\mb{P}^b,\on{Span}(b,b))\backslash \left(\Phi^{\mb{P}^b,a+r-b}_{d_1,\ldots,d_{k-1}}(\mb{P}^b)\times W_{b,d_k}\right)
\end{align}
in $\prod_{i=1}^{k}{W_{b,d_i}}$. Therefore, the codimension of \eqref{tlb} is the codimension of \eqref{ie2} minus the dimension of $\mb{G}(b,r)$. 

To compute the codimension of \eqref{ie2} in $\prod_{i=1}^{k}{W_{b,d_i}}$, it suffices to prove the statement of \Cref{technical} when $b=r$, so we want to bound the codimension of 
\begin{align}
\label{tlb2}
    \Phi^{\mb{P}^r,a}_{d_1,\ldots,d_k}(\mb{P}^r,\on{Span}(r,r))\backslash \left(\Phi^{\mb{P}^r,a}_{d_1,\ldots,d_{k-1}}(\mb{P}^r)\times W_{r,d_k}\right)
\end{align} 
in $\prod_{i=1}^{k}{W_{r,d_i}}$ from below. To do this, suppose $(F_1,\ldots,F_k)$ is in \eqref{tlb2}. Then, there exist $a$ values (possibly more) $1\leq i_1<\cdots<i_a=k$ where for $i=i_j$, 
$F_{i_j}$ must vanish on one of the nondegenerate components $Y$ of $\{F_1=\cdots=F_{i_j-1}=0\}$. Since that component $Y$ is dimension at least $r-i_j+j$, it is at least $h_{r,r-i_j+j}(d_{i_j})$ conditions for $F_{i_j}$ to vanish on $Y$. Taking this over all possibilities for $1\leq i_1<\cdots<i_a=k$ yields the bound
\begin{align}
\label{ndb}
    \min\left\{\sum_{j=1}^{a}h_{r,r-i_j+j}(d_{i_j})\mid 1\leq i_1<\cdots<i_{a}= k\right\},
\end{align}
which is what we wanted. 
\end{proof}

\begin{remark}
The part of the proof of \Cref{technical} that can be made more precise is the very end, where we argue there are $a$ instances where the next hypersurface must contain a nondegenerate component of the intersection of the previous hypersurfaces and thus obtain a bound on the codimension of \eqref{tlb2}. The arguments in \cite{TCIH} make this part rigorous, but we choose not to repeat the arguments here for the sake of clarity and length.

If the reader wants a more detailed proof of the bound \eqref{ndb}, see \cite[Lemma 4.2]{TCIH} for the lower bound for the codimension of $\Phi^{\mb{P}^r,a}_{d_1,\ldots,d_k}(\mb{P}^r,\on{Span}(r,r))$ in $\prod_{i=1}^{k}{W_{r,d_i}}$ given as 
\begin{align*}
    \min\left\{\sum_{j=1}^{a}h_{r,r-i_j+j}(d_{i_j})\mid 1\leq i_1<\cdots<i_{a}\leq k\right\}.
\end{align*}
The only difference is we minimized over all $1\leq i_1<\cdots<i_a\leq k$ instead of over $1\leq i_1<\cdots<i_a=k$. See also \cite[Section 2]{TCIH} for an example of how the argument works without the heavy notation.
\end{remark}

\begin{proof}[Proof of \Cref{MT2}]
By setting up the usual incidence correspondence, 
\begin{center}
    \begin{tikzcd}
     &\{(F_1,\ldots,F_{r+a-1}),\ell\mid F_i\text{ restricts to 0 on the line }\ell\text{ for all }i\} \ar[dl] \ar[dr] &   \\
     \mb{G}(1,r) & & \prod_{i=1}^{r+a-1}W_{r,d_i}
    \end{tikzcd}
\end{center}
one sees that the locus of forms $(F_1,\ldots,F_{r+a-1})$ where $\{F_1=\cdots=F_{r+a-1}=0\}$ contains some line (e.g. $\Phi^{\mb{P}^r,a}_{d_1,\ldots,d_{r+a-1}}(\mb{P}^r,\on{Span}(r,1))$) is codimension 
\begin{align}\label{b1}
    -\dim(\mb{G}(1,r))+\sum_{i=1}^{r+a-1}{(d_i+1)}
\end{align}
in $\prod_{i=1}^{r+a-1}W_{r,d_i}$.
Now, we want to apply \Cref{technical} and show the codimension of 
\begin{align}\label{bound}
    \Phi^{\mb{P}^r,a}_{d_1,\ldots,d_{r+a-1}}(\mb{P}^r,\on{Span}(r,b))\backslash \left(\Phi^{\mb{P}^r,a}_{d_1,\ldots,d_{r+a-2}}(\mb{P}^r)\times W_{r,d_{r+a-1}}\right)
\end{align} 
is always greater than the codimension of $\Phi^{\mb{P}^r,a}_{d_1,\ldots,d_{r+a-1}}(\mb{P}^r,\on{Span}(r,1))$ if $b>1$. 

From \Cref{technical}
\begin{align}
\label{bt2}
       -\dim(\mb{G}(b,r))+\min\left\{\sum_{j=1}^{r-b+a}h_{b,b-i_j+j}(d_{i_j})\mid 1\leq i_1<\cdots<i_{r-b+a}= r+a-1\right\}
\end{align}
is a lower bound for the codimension of \eqref{bound}. We can bound \eqref{bt2} from below by
\begin{align}
\label{bt22}
        -\dim(\mb{G}(b,r))+\min\left\{\sum_{j=1}^{r-b+a}h_{b,1}(d_{i_j})\mid 1\leq i_1<\cdots<i_{r-b+a}= r+a-1\right\},
\end{align}
which is
\begin{align*}
    -\dim(\mb{G}(b,r))+\sum_{j=1}^{r-b+a-1}(bd_j+1)+(bd_{r+a-1}+1)&\geq \\
    -\dim(\mb{G}(b,r))+(-b+1) + \sum_{j=1}^{r+a-1}(d_j+1)+\sum_{j=1}^{r-b+a-1}((b-1)d_j)&\geq \\
    -(b+1)(r-b)+(-b+1)+(b-1)(r-b+a-1)+ \sum_{j=1}^{r+a-1}(d_j+1)&=\\
    -2(r-1)+a(b-1)+ \sum_{j=1}^{r+a-1}(d_j+1)&,
\end{align*}
which is always greater than \eqref{b1} since $b>1$.
\end{proof}

\bibliographystyle{alpha} \bibliography{references.bib}
\end{document}